\documentclass[12pt,reqno]{amsart}
\usepackage{amsmath}
\usepackage{amssymb,enumitem}
\usepackage{mathtools}
\usepackage{hyperref}
\usepackage{faktor}
\usepackage{todonotes}
\usepackage[
  backend=biber,
  style=alphabetic,
  url=false,
  isbn=false
]{biblatex}
\usepackage{ifthen}

\newtheorem {Theorem}    {Theorem}[section]
\newenvironment{Theorem*}
  {\Theorem}
  {\endTheorem}
\newtheorem {Lemma}      [Theorem]    {Lemma}

\newtheorem {Proposition}[Theorem]    {Proposition}

\theoremstyle{definition}
\newtheorem {Definition} [Theorem]    {Definition}
\newtheorem {Example}    [Theorem]    {Example}
\newtheorem {Remark}    [Theorem]    {Remark}
\newtheorem {Exercise}   {Exercise}
\newenvironment{Exercise*}
  {\Exercise}
  {\endExercise}
\newtheorem {Question}   [Theorem]    {Question}
\newcounter{AbcT}

\numberwithin{equation}{section}

\newcommand {\N} {{\mathbb N}}

\newcommand {\T} {{\mathbb T}}

\newcommand {\Z} {{\mathbb Z}}

\newcommand {\cU} {{\mathcal U}}

\renewcommand{\liminf}{\varliminf}
\renewcommand{\limsup}{\varlimsup}

\DeclareMathOperator{\mdim}{mdim}

\newcommand {\IGNORE}[1]  {}

\renewcommand {\setminus}       {\smallsetminus}

\newcommand{\myentry}[1]{\ifthenelse{\equal{#1}{.}}{}{#1}}

\DeclareMathOperator{\Spoke}{Spoke}
\title{Uniformly positive mean dimension}
\author{Tal Barak}
\author{Elon Lindenstrauss}
\thanks{The authors acknowledge support by ERC advanced grant HomDyn (grant number 833423).}
\date{\today}

\begin{document}

\begin{abstract}
We study the relation between uniformly positive entropy and uniformly positive
mean dimension at the level of fixed open covers. To a symbolic system \(X\),
we associate a hub-and-spoke system \(\operatorname{Spoke}(X)\), obtained by
replacing each symbol by a one-dimensional spoke attached to a common hub. We
prove that if \(X\) admits a shift-invariant measure of full support, then
\(\operatorname{Spoke}(X)\) has completely positive mean dimension. We also
prove that if \(X\) has uniformly positive entropy, then \(\operatorname{Spoke}(X)\)
has uniformly positive mean dimension. Finally, using symbolic codings of
irrational rotations on tori, we construct hub-and-spoke systems with completely
positive mean dimension but without uniformly positive mean dimension or
uniformly positive entropy. The examples are nondegenerate: the relevant covers
have zero mean dimension and zero entropy, but when refined by iterating under the dynamics the corresponding covering numbers are unbounded.
\end{abstract}

\maketitle
\setcounter{tocdepth}{1}
\tableofcontents
\section{Introduction}

Mean dimension is an invariant introduced by Gromov to study certain
infinite-dimensional dynamical systems. In~\cite{Lindenstrauss-Weiss},
Weiss and the second named author linked mean dimension to topological
entropy.

Suppose \((X,T)\) is a topological dynamical system, i.e. \(X\) is a compact
metrizable space and \(T:X\to X\) is a homeomorphism. If \(\cU\) is an open
cover of \(X\), we let \(\mathcal D(\cU)\) denote the dimension captured by
\(\cU\), i.e. the minimal dimension of a CW-complex \(Y\) such that there is a
continuous map
\[
f:X\to Y
\]
and an open cover \(\widetilde{\cU}\) of \(Y\) such that
\(f^{-1}(\widetilde{\cU})\) refines \(\cU\).
More concretely, one can define
\[
\mathcal D (\cU) = \min\operatorname{ord} (\gamma)
\]
over all open covers $\gamma$ refining $\cU$, with $\operatorname{ord}(\gamma) + 1$ being the maximal multiplicity of the cover $\gamma$.

 One easily sees
from the definition that for any two covers \(\cU,\mathcal V\),
\[
\mathcal D(\cU\vee \mathcal V)
\leq
\mathcal D(\cU)+\mathcal D(\mathcal V).
\]
Hence, if
\[
\cU_a^b=\bigvee_{i=a}^b T^{-i}\cU,
\]
then the limit
\[
\lim_{n\to\infty}
\frac{\mathcal D(\cU_0^{n-1})}{n}
\]
exists. Set
\[
\mdim(X,T)
=
\sup_{\cU}
\lim_{n\to\infty}
\frac{\mathcal D(\cU_0^{n-1})}{n},
\]
where the supremum is taken over all finite open covers of \(X\). It will be
convenient to denote
\[
\mdim(X,T,\cU)
=
\lim_{n\to\infty}
\frac{\mathcal D(\cU_0^{n-1})}{n}.
\]

Recall also that topological entropy is defined by
\[
h_{\mathrm{top}}(X,T)
=
\sup_{\cU}
\lim_{n\to\infty}
\frac{\log \mathcal N(\cU_0^{n-1})}{n},
\]
where \(\mathcal N(\cU)\) denotes the
minimal cardinality of a subcover of \(\cU\) covering \(X\).
Again, the inner limit exists by subadditivity; we denote it by
\[
h_{\mathrm{top}}(X,T,\cU).
\]

Let \(d\) be a metric on \(X\) compatible with the topology. For an open cover
\(\cU\) of \(X\), set
\[
\operatorname{mesh}(\cU,d)
=
\max_{U\in\cU}\operatorname{diam}(U),
\]
and let
\[
N(X,d,\epsilon)
=
\min\{|\cU|:\cU \text{ is an open cover of }X
\text{ with }\operatorname{mesh}(\cU,d)<\epsilon\}.
\]
Now set
\[
d_N(x,y)=\max_{0\leq n<N}d(T^n x,T^n y)
\]
and
\[
S(X,T,d,\epsilon)
=
\lim_{n\to\infty}
\frac{\log N(X,d_n,\epsilon)}{n}.
\]
If
\[
\underline{\mdim}_M(X,T,d)
=
\liminf_{\epsilon\searrow0}
\frac{S(X,T,d,\epsilon)}{\log(1/\epsilon)},
\]
then it was shown in~\cite{Lindenstrauss-Weiss} that for any compatible metric
\(d\),
\[
\mdim(X,T)
\leq
\underline{\mdim}_M(X,T,d).
\]
It was shown in~\cite{Lindenstrauss-mean-dimension} that if \((X,T)\) has the
marker property, in particular if \((X,T)\) has an infinite minimal factor,
then there is a compatible metric \(d\) such that
\[
\mdim(X,T)=\underline{\mdim}_M(X,T,d).
\]
Both results relating mean dimension and entropy were extended by Tsukamoto
and the second named author in~\cite{Lindenstrauss-Tsukamoto-GAFA}; see also
\cite{Lindenstrauss-Tsukamoto-IEEE}. In particular, under the marker property
there is a compatible metric \(d\) such that
\[
\mdim(X,T)=\overline{\mdim}_M(X,T,d),
\]
where
\[
\overline{\mdim}_M(X,T,d)
=
\limsup_{\epsilon\searrow0}
\frac{S(X,T,d,\epsilon)}{\log(1/\epsilon)}.
\]

While these results give a satisfying answer regarding the relation between
mean dimension and entropy, they do not work at the level of a fixed cover:
they do not directly relate \(\mdim(X,T,\cU)\) and
\(h_{\mathrm{top}}(X,T,\cU)\) for a fixed \(\cU\). The aim of this paper is to study the relation between uniformly positive
entropy and uniformly positive mean dimension.  We focus on spoke systems
associated to symbolic dynamics, where the entropy-theoretic independence
structure of the symbolic base can be converted into mean-dimensional
positivity of the corresponding spoke system.

In his paper \cite{blanchard-upe-cpe}, Blanchard considered two different notions that can be thought of as a topological analogue of the ergodic-theoretic class of
\(K\)-systems, i.e. systems with no zero entropy factors. Here and
below, by a standard open cover we mean a two-element open cover
\(\{U,V\}\) such that neither \(U\) nor \(V\) is dense.

\begin{Definition}[\cite{blanchard-upe-cpe}]
\begin{enumerate}
    \item A topological dynamical system has Completely Positive Entropy
    (CPE) if it has no nontrivial topological zero entropy factors.

    \item A topological dynamical system has Uniformly Positive Entropy
    (UPE) if for every standard open cover \(\cU\) of \(X\), one has
    \[
    h_{\mathrm{top}}(X,T,\cU)>0.
    \]
\end{enumerate}
\end{Definition}

Clearly UPE implies CPE, but Blanchard shows in
\cite{blanchard-upe-cpe} that the converse need not hold. In their paper
\cite{Garcia-Gutman-mdim}, Garc\'ia-Ramos and Gutman introduced analogous
properties for mean dimension, in a much wider context than the one considered
here, namely actions of general sofic groups.

\begin{Definition}
\begin{enumerate}
    \item A topological dynamical system has Completely Positive Mean Dimension
    (CPMD) if it has no nontrivial topological zero mean dimension factors.

    \item A topological dynamical system has Uniformly Positive Mean Dimension
    (UPMD) if for every standard open cover \(\cU\) of \(X\), one has
    \[
    \mdim(X,T,\cU)>0.
    \]
\end{enumerate}
\end{Definition}

\noindent
These notions also arose independently in the first named author's MSc thesis.

Again, UPMD implies CPMD. Garc\'ia-Ramos and Gutman give an example of a CPMD
system that is not UPMD. While this is not stated explicitly, mainly because
the concept of mean dimension was not yet available, an example in
\cite{Lindenstrauss-lowering-entropy} gives a minimal system with CPMD.

Garc\'ia-Ramos and Gutman raise the following question.

\begin{Question}\label{question 1}
Does every UPMD system have UPE?
\end{Question}

Since finite entropy systems have zero mean dimension, CPMD implies CPE;
indeed, every nontrivial factor of a CPMD system has infinite topological
entropy. The relation between UPMD and UPE is more subtle.

We investigate this question for a particular class of examples, which we call the `hub-and-spoke construction' (here spoke is used as a noun, i.e.\ a spoke of a wheel). The construction is built from a symbolic system, which plays the role of the `base'. These examples may be constructed as follows:
Let \((X,\sigma)\) be a symbolic shift,
i.e. \(X\) is a closed shift-invariant subset of
\[
\{1,\ldots,a\}^{\mathbb Z}
\]
for some \(a\in\mathbb N\). Let \(\overline{\mathbb D}\) be the closed unit disk
in \(\mathbb R^2\). Choose distinct points
\[
p_1,\ldots,p_{a}\in\partial\mathbb D.
\]
Define
\[
\Phi:X\times[0,1]^{\mathbb Z}\longrightarrow\overline{\mathbb D}^{\mathbb Z}
\]
by
\[
\Phi(x,t)_i=t_i p_{x_i}.
\]
We write
\[
\Spoke(X)=\Phi\bigl(X\times[0,1]^{\mathbb Z}\bigr).
\]
Then \(\Spoke(X)\) is a closed, shift-invariant subsystem of
\((\overline{\mathbb D}^{\mathbb Z},\sigma)\). It contains the constant zero
point as a fixed point, and, for nonempty \(X\), has mean dimension \(1\).

\subsection*{Remarks:}
\begin{enumerate}
    \item When \(X=\{0,1,2\}^{\mathbb Z}\), the resulting system \(\Spoke(X)\), and
large minimal subshifts in it, were used by Tsukamoto and the second named
author in~\cite{Lindenstrauss-Tsukamoto} (without using this terminology) to give examples of minimal systems
of mean dimension \(\frac12+\epsilon\) that cannot be embedded in
\[
\bigl([0,1]^{\mathbb Z},\sigma\bigr).
\]
\item Isomorphic symbolic systems can give rise to non-isomorphic spoke systems. Indeed, let
\[
X=\{0,1\}^{\mathbb Z}
\]
be the binary full shift, and let $X'$ be the image of $X$ in $\{00,01,10,11\}^{\mathbb Z}$ under the map
\[
(x_i)_{i\in\mathbb Z} \mapsto (x_i\,x_{i+1})_{i\in\mathbb Z}.
\]
This map is clearly an isomorphism, but the fixed point set for $\sigma^2$ in $\Spoke(X)$ is homeomorphic to a square, whereas in $\Spoke(X')$ it is homeomorphic to four squares glued together in one corner.
\item The example by Garc\'ia-Ramos and
Gutman in~\cite{Garcia-Gutman-mdim} of a CPMD system that is not UPMD given can be viewed as a hub-and-spoke
system. 
\end{enumerate}

Unless the symbolic system is rather degenerate, a hub-and-spoke system will have CPMD. Indeed, we prove the following:

\begin{Theorem}
\label{Invariant measure of full support implies CPMD}
Let \((X,\sigma)\) be a symbolic system that has a shift-invariant measure
\(\mu\) of full support. Then \((\Spoke(X),\sigma)\) has CPMD.
\end{Theorem}

\noindent
Note that we do not assume the measure of full support to be ergodic, nor that $X$ is infinite.

\medskip

Whether a hub-and-spoke system has UPMD is more delicate. Our positive result is the following: if the symbolic base has UPE, then the resulting hub-and-spoke system has UPMD.

\begin{Theorem}
\label{UPE symbolic implies spoke UPMD}
Let \((X,\sigma)\) be a symbolic system with UPE. Then the corresponding hub-and-spoke system 
\( (\Spoke(X),\sigma) \)
has UPMD.
\end{Theorem}

It is easy to see that if the symbolic base fails to have UPE, the hub-and-spoke system will also fail UPE, so in this case UPE implies UPMD. We recall that by a result of Glasner and Weiss \cite{Glasner-Weiss-Strictly-ergodic-uniform-positive-entropy-models}, systems for which there is a fully supported invariant measure which is $K$ have UPE; in particular, symbolic systems with UPE exist.

As mentioned above, CPMD alone does not force UPMD.
However, the example in~\cite{Garcia-Gutman-mdim} of a system \((Z,T)\) with
CPMD but not UPMD is somewhat unsatisfactory in that the open cover \(\cU\) for
which
\[
\mdim(Z,T,\cU)=0
\]
is dynamically degenerate:
\[
\mathcal N(\cU_0^{n-1})=\mathcal N(\cU)
\qquad
(n\geq1).
\]
Our final family of examples are hub-and-spoke
systems with CPMD but without UPMD, witnessed by a standard two-set cover
\(\cU\) for which
\[
\mdim(Z,\sigma,\cU)=0,
\]
but so that the covering number $\mathcal N(\cU_0^{n-1})\to\infty$.
These systems are based on a symbolic coding of irrational rotation on tori.

\subsection*{Acknowledgments}
This paper is based on the first named author's M.Sc.\ thesis, conducted under the guidance of the second named author. The authors would like to thank Benjy Weiss for his interest and encouragement to submit these results as a research paper. We also thank Eli Glasner for bringing to our attention the results of \cite{Garcia-Gutman-mdim}.

\section{Existence of an invariant measure of full support implies CPMD}
In this section we will present the proof of Theorem \ref{Invariant measure of full support implies CPMD}.
Recall that this theorem states that if $(X,\sigma)$ is a symbolic system with a shift-invariant measure $\mu$ of full support, then $(\Spoke(X),\sigma)$ has CPMD. 

Before presenting the proof, we need to recall the following standard result in dimension theory:

\begin{Lemma}
\label{bound of covering number of with no opposing faces}
If $\cU$ is a finite open cover of $[0,1]^d$ such that no $U\in \cU$ intersects two opposing faces of the cube, then $\operatorname{ord}(\cU)\geq d$.
\end{Lemma}

For proof see e.g.\ {\cite[Lem.\ 3.2]{Lindenstrauss-Weiss}}. We can now prove Theorem~\ref{Invariant measure of full support implies CPMD}:

\begin{proof}[Proof of Theorem~\ref{Invariant measure of full support implies CPMD}]
Let $Z=\Spoke(X)$, and suppose that there exists some nontrivial factor
\[
\varphi:(Z,\sigma)\rightarrow (Q,S)
\]
where $(Q,S)$ is a topological dynamical system with zero mean dimension.

Let $\underline{0}$ be the constant zero sequence in $Z$. Since $\varphi$ is not constant, there exists $z'\in Z$ such that
\[
\varphi(z')\neq \varphi(\underline{0}).
\]
Let $g:Q\rightarrow [0,1]$ be a continuous function such that
\[
g(\varphi(\underline{0}))=0,\qquad g(\varphi(z'))=1.
\]
Define
\[
\tilde g=g\circ\varphi:Z\rightarrow [0,1].
\]
Thus
\[
\tilde g(\underline{0})=0,\qquad \tilde g(z')=1.
\]

Write $z'=\Phi(x',r')$ with $x'\in X$ and $r'\in[0,1]^\Z$. Thus, if the spoke directions are denoted by $p_1,\ldots,p_{a}$, then
\[
z'_j=r'_jp_{x'_j}\qquad (j\in\Z).
\]
By the continuity of $\tilde g$ there is an $N\in\N$ such that if
\[
w|_{[-N,N]}=z'|_{[-N,N]},
\]
then $\tilde g(w)>0.9$, and if
\[
w|_{[-N,N]}=\underline{0}|_{[-N,N]},
\]
then $\tilde g(w)<0.1$.

Let $\mathcal P$ be the open cover of $Q$
\[
\mathcal P=\left\{\{\xi\in Q:g(\xi)>0.4\},\{\xi\in Q:g(\xi)<0.6\}\right\},
\]
and set
\[
\tilde{\mathcal P}:=\varphi^{-1}(\mathcal P)=\{\varphi^{-1}(P):P\in\mathcal P\}.
\]
Since $\tilde{\mathcal P}=\varphi^{-1}(\mathcal P)$, for every $M$ we have
\[
\mathcal D(\tilde{\mathcal P}_0^{M-1})\leq \mathcal D(\mathcal P_0^{M-1}).
\]
Hence
\[
\mdim(Z,\sigma,\tilde{\mathcal P})\leq \mdim(Q,S,\mathcal P)\leq \mdim(Q,S)=0.
\]

Let
\[
A=\{x\in X:x_j=x'_j\text{ for all }-N\leq j\leq N\}.
\]
This is a nonempty open cylinder in $X$. Since $\mu$ has full support,
\[
\theta:=\mu(A)>0.
\]

For $M\in\N$, define
\[
F_M(x)=\sum_{n=0}^{M-1}1_A(\sigma^n x).
\]
By the $\sigma$-invariance of $\mu$,
\[
\int_X F_M(x)\,d\mu(x)=\sum_{n=0}^{M-1}\mu(\sigma^{-n}A)=M\mu(A)=\theta M.
\]
Therefore there exists $x^{(M)}\in X$ such that $F_M(x^{(M)})\geq \theta M$. Equivalently, if
\[
J_M=\{0\leq n<M:\sigma^n x^{(M)}\in A\},
\]
then
\[
|J_M|\geq \theta M.
\]

Choose a subset $\{n_1<\cdots<n_k\}\subset J_M$ iteratively as follows. Choose $n_1$ to be the smallest element of $J_M$, then remove from $J_M$ all points in $[n_1,n_1+2N+1]$, and repeat. Then
\[
n_{i+1}-n_i>2N+1\qquad (1\leq i<k),
\]
so the intervals $[n_i-N,n_i+N]$, $1\leq i\leq k$, are pairwise disjoint. Moreover, each chosen point removes at most $2N+2$ points of $J_M$, hence
\[
k\geq \frac{|J_M|}{2N+2}-1\geq \frac{\theta M}{2N+2}-1.
\]
Thus, for all sufficiently large $M$, we have
\[
k\geq cM
\]
for some constant $c>0$ independent of $M$, for example
\[
c=\frac{\theta}{2(2N+2)}.
\]

For every $1\leq i\leq k$, since $n_i\in J_M$ we have $\sigma^{n_i}x^{(M)}\in A$, and therefore
\[
x^{(M)}_{n_i+j}=x'_j\qquad (-N\leq j\leq N).
\]

Consider now the map $\Psi:[0,1]^k\rightarrow Z$ given coordinatewise by
\[
\Psi(\alpha)_\ell=
\begin{cases}
0, & \text{if there is no }1\leq i\leq k\text{ such that }|\ell-n_i|\leq N,\\
\alpha_i z'_j, & \text{if }\ell=n_i+j\text{ for some }1\leq i\leq k\text{ and }|j|\leq N.
\end{cases}
\]
This is well-defined because the intervals $[n_i-N,n_i+N]$ are pairwise disjoint.

We check that $\Psi(\alpha)\in Z$. Define $s=s(\alpha)\in[0,1]^\Z$ by
\[
s_{n_i+j}=\alpha_i r'_j\qquad (-N\leq j\leq N,\ 1\leq i\leq k),
\]
and $s_\ell=0$ outside the selected intervals. If $\ell=n_i+j$, then
\[
\Psi(\alpha)_\ell=\alpha_i z'_j=\alpha_i r'_jp_{x'_j}=s_\ell p_{x^{(M)}_\ell},
\]
because $x^{(M)}_{n_i+j}=x'_j$. Outside the selected intervals we have
\[
\Psi(\alpha)_\ell=0=s_\ell p_{x^{(M)}_\ell}.
\]
Therefore
\[
\Psi(\alpha)=\Phi(x^{(M)},s)\in Z.
\]
It is also clear from the coordinate definition that $\Psi$ is continuous.

Now fix $i\in\{1,\ldots,k\}$. If $\alpha_i=0$, then the block of $\sigma^{n_i}\Psi(\alpha)$ on $[-N,N]$ is the zero block, hence
\[
\tilde g(\sigma^{n_i}\Psi(\alpha))<0.1.
\]
If $\alpha_i=1$, then the block of $\sigma^{n_i}\Psi(\alpha)$ on $[-N,N]$ is exactly $z'|_{[-N,N]}$, hence
\[
\tilde g(\sigma^{n_i}\Psi(\alpha))>0.9.
\]

In other words,
\[
\Psi^{-1}\left(\bigvee_{i=1}^k\sigma^{-n_i}\tilde{\mathcal P}\right)
\]
is an open cover of the cube $[0,1]^k$ such that no element intersects two opposing faces. Indeed, at the $i$-th coordinate, the set corresponding to $\{\tilde g>0.4\}$ misses the face $\{\alpha_i=0\}$, while the set corresponding to $\{\tilde g<0.6\}$ misses the face $\{\alpha_i=1\}$.

Every finite open refinement of this pulled-back cover has the same property. Hence by Lemma \ref{bound of covering number of with no opposing faces},
\[
\mathcal D\left(\Psi^{-1}\left(\bigvee_{i=1}^k\sigma^{-n_i}\tilde{\mathcal P}\right)\right)\geq k.
\]
By the continuity of $\Psi$,
\[
\mathcal D\left(\Psi^{-1}\left(\bigvee_{i=1}^k\sigma^{-n_i}\tilde{\mathcal P}\right)\right)
\leq
\mathcal D\left(\bigvee_{i=1}^k\sigma^{-n_i}\tilde{\mathcal P}\right).
\]
Therefore
\[
\mathcal D\left(\bigvee_{i=1}^k\sigma^{-n_i}\tilde{\mathcal P}\right)\geq k.
\]

Since $\{n_1,\ldots,n_k\}\subset\{0,\ldots,M-1\}$, the full join
\[
\tilde{\mathcal P}_0^{M-1}=\bigvee_{m=0}^{M-1}\sigma^{-m}\tilde{\mathcal P}
\]
refines
\[
\bigvee_{i=1}^k\sigma^{-n_i}\tilde{\mathcal P}.
\]
By monotonicity of $\mathcal D$ under refinements,
\[
\mathcal D(\tilde{\mathcal P}_0^{M-1})\geq \mathcal D\left(\bigvee_{i=1}^k\sigma^{-n_i}\tilde{\mathcal P}\right)\geq k\geq cM
\]
for all sufficiently large $M$. Thus
\[
\mdim(Z,\sigma,\tilde{\mathcal P})=\lim_{M\rightarrow\infty}\frac{\mathcal D(\tilde{\mathcal P}_0^{M-1})}{M}\geq c>0.
\]
This contradicts $\mdim(Z,\sigma,\tilde{\mathcal P})= 0$. Hence no nontrivial zero mean dimension factor exists, and therefore $(Z,\sigma)$ has CPMD.
\end{proof}

\begin{Remark}
Theorem~\ref{Invariant measure of full support implies CPMD} might suggest at
first that the spoke construction automatically tends to produce CPMD systems.
Indeed, the theorem applies even to very degenerate symbolic systems, such as
\[
X=\{0^{\mathbb Z}\}
\qquad\text{or}\qquad
X=\{0^{\mathbb Z},1^{\mathbb Z}\},
\]
where the shift acts as the identity and \(X\) carries an invariant measure of
full support.  The following example shows that some condition on the base is nevertheless needed:
without the full-support hypothesis on the symbolic base, a spoke system need
not have CPMD.
\end{Remark}

\begin{Example}
Fix \(k\geq 1\), and let
\[
Y_k=
\left\{
y\in\{0,1\}^{\mathbb Z}:
\#\{m\in\mathbb Z:y_m=1\}\leq k
\right\}.
\]
It is easy to check that \(Y_k\) is a closed shift-invariant subsystem of
\(\{0,1\}^{\mathbb Z}\). Let
\[
Z_k=\operatorname{Spoke}(Y_k)
\]
with spoke directions \(\xi_0,\xi_1\). We show that \(Z_k\) does not have
CPMD.

Let
\[
S=\xi_0[0,1]\cup\xi_1[0,1].
\]
Define
\[
\rho_1:S\to[0,1]
\]
by
\[
\rho_1(t\xi_1)=t,
\qquad
\rho_1(t\xi_0)=0.
\]
Thus \(\rho_1\) records the radius on the \(\xi_1\)-spoke and vanishes on the
\(\xi_0\)-spoke. It is continuous, since the two formulas agree at the origin.
Define
\[
\Pi_1:Z_k\to[0,1]^{\mathbb Z},
\qquad
(\Pi_1(z))_m=\rho_1(z_m);
\]
then \(\Pi_1\) is continuous, and
shift-equivariant.

The image of \(\Pi_1\) is exactly
\[
Q_k=
\left\{
u\in[0,1]^{\mathbb Z}:
\#\{m\in\mathbb Z:u_m\neq0\}\leq k
\right\}.
\]
Indeed, a nonzero coordinate of \(\Pi_1(z)\) can occur only where the symbolic
point \(y\in Y_k\) has symbol \(1\), and this happens at most \(k\) times.
Conversely, every \(u\in Q_k\) is obtained by putting its nonzero coordinates on
the \(\xi_1\)-spoke and all other coordinates at the hub.

We claim that
\[
\operatorname{mdim}(Q_k,\sigma)=0.
\]
Let \(\beta\) be a finite open cover of \(Q_k\). Since \(Q_k\) is compact and
has the product topology inherited from \([0,1]^{\mathbb Z}\), there is
\(M\in\mathbb N\) such that
\[
\pi_{[-M,M]}:Q_k\to[0,1]^{[-M,M]}
\]
is \(\beta\)-compatible. Hence, for
\[
I_n=[-M,n-1+M]\cap\mathbb Z,
\]
the projection
\[
\pi_{I_n}:Q_k\to\pi_{I_n}(Q_k)
\]
is \(\beta_0^{n-1}\)-compatible.

Now \(\pi_{I_n}(Q_k)\) consists of all points in \([0,1]^{I_n}\) with at most
\(k\) nonzero coordinates. Equivalently,
\[
\pi_{I_n}(Q_k)
=
\bigcup_{\substack{J\subset I_n\\ |J|\leq k}}
\left\{
x\in[0,1]^{I_n}:x_\ell=0\text{ for every }\ell\notin J
\right\}.
\]
Each set in this finite union is naturally homeomorphic to the cube
\([0,1]^J\), and therefore has dimension \(|J|\leq k\). Hence
\[
\dim \pi_{I_n}(Q_k)\leq k.
\]
Therefore
\[
\mathcal D(\beta_0^{n-1})\leq k
\qquad(n\geq1),
\]
and so
\[
\operatorname{mdim}(Q_k,\sigma,\beta)=0.
\]
Since \(\beta\) was arbitrary,
\[
\operatorname{mdim}(Q_k,\sigma)=0.
\]

The factor \(Q_k\) is nontrivial. Therefore \(Z_k\) has a nontrivial zero mean
dimension factor. Hence \(Z_k\) does not have CPMD.
\end{Example}
\section{UPE symbolic systems imply UPMD for hub-and-spoke systems}

Before proving the theorem, we collect the independence terminology and the
external result needed for the argument. The only entropy theoretic input is
the Kerr--Li characterization of uniformly positive entropy in terms of
positive-density independence sets.

\begin{Definition}[Independence set]
Let \((X,T)\) be a topological dynamical system, and let
\[
\mathbf A=(A_1,A_2,\ldots,A_k)
\]
be a tuple of subsets of \(X\). We say that a set \(J\subset\mathbb Z\)
is an independence set for \(\mathbf A\) if for every nonempty finite
subset \(I\subset J\) and every function
\[
\omega:I\to\{1,2,\ldots,k\},
\]
we have
\[
\bigcap_{n\in I}T^{-n}A_{\omega(n)}\neq\varnothing.
\]
\end{Definition}

Equivalently, \(J\) is an independence set for \(\mathbf A\) if, whenever
we choose finitely many times from \(J\), we may independently prescribe
which set \(A_i\) the orbit should visit at each chosen time, and there is
some point \(x\in X\) realizing all these prescriptions. 
We say that \(J\subset\mathbb Z_{\geq0}\) has positive upper density if
\[
\overline d(J):=
\limsup_{M\to\infty}
\frac{|J\cap\{0,\ldots,M-1\}|}{M}>0.
\]

\begin{Definition}[IE-tuple]
Let \((X,T)\) be a topological \(\mathbb Z\)-system. A tuple
\[
x=(x_1,\ldots,x_k)\in X^k
\]
is called an IE-tuple if for every product neighbourhood
\[
U_1\times\cdots\times U_k
\]
of \(x\), the tuple of sets
\[
(U_1,\ldots,U_k)
\]
has an independence set of positive upper density.
\end{Definition}

\begin{Definition}[Entropy tuple]
Let \((X,T)\) be a topological \(\mathbb Z\)-system. For \(k\geq2\), a tuple
\[
x=(x_1,\ldots,x_k)\in X^k\backslash\Delta_k(X),
\qquad
\Delta_k(X)=\{(x,\ldots,x):x\in X\},
\]
is called an entropy tuple if whenever
\[
\bar U_1,\ldots,\bar U_\ell
\]
are closed pairwise disjoint neighbourhoods in \(X\) of the distinct points
appearing in the list \(x_1,\ldots,x_k\), the open cover
\[
\{\bar U_1^c,\ldots,\bar U_\ell^c\}
\]
has positive topological entropy.
\end{Definition}

\begin{Remark}
Entropy tuples are the higher-order analogue of entropy pairs. When
\(k=2\), an entropy tuple is exactly an entropy pair.
\end{Remark}

\begin{Theorem}[Kerr--Li {\cite[Theorem~3.16]{Kerr-Li-Independence}}]
\label{kerr-li-entropy-ie}
Let \((X,T)\) be a topological dynamical system. Let
\[
(x_1,\ldots,x_k)\in X^k\backslash\Delta_k(X),
\qquad k\geq 2.
\]
Then \((x_1,\ldots,x_k)\) is an entropy tuple if and only if it is an
IE-tuple.
\end{Theorem}

As an immediate corollary of Theorem~\ref{kerr-li-entropy-ie}, if
\((X,T)\) has UPE, then for every pair of nonempty open sets
\(A_0,A_1\subset X\), there exists an independence set
\[
J\subset\mathbb Z_{\geq0}
\]
for \((A_0,A_1)\) with positive upper density.

With this consequence in hand, we are now ready to prove the main theorem.
% \begin{Theorem}
% \label{UPE symbolic implies spoke UPMD}
% Let \((X,\sigma)\) be a symbolic system with UPE. Then
% \[
% (\Spoke(X),\sigma)
% \]
% has UPMD.
% \end{Theorem}

\begin{proof}[Proof of Theorem~\ref{UPE symbolic implies spoke UPMD}]
Let
\[
Z = \Spoke(X).
\]
We prove that every standard open cover of \(Z\) has positive mean dimension.
Let
\[
\cU = \{U,V\}
\]
be a standard open cover of \(Z\). Since \(U\) and \(V\) are non-dense open
subsets of \(Z\), choose
\[
z^0 \in Z \backslash \overline{U},
\qquad
z^1\in Z \backslash \overline{V}.
\]
Because \(Z \subset \overline{\mathbb D}^{\mathbb Z}\) is equipped with the
induced product topology, there exists \(N \in \mathbb N\), and open sets
\[
O^0,O^1 \subset \overline{\mathbb D}^{[-N,N]},
\]
such that the neighborhoods
\[
C^0 := Z \cap  \pi^{-1}_{[-N,N]}(O^0),
\qquad
C^1 := Z \cap  \pi^{-1}_{[-N,N]}(O^1)
\]
satisfy
\[
z^0 \in C^0 \subset Z \backslash U ,
\qquad
z^1 \in C^1\subset Z \backslash V .
\]

By the definition of
\[
Z=\Spoke(X):=\Phi\bigl(X\times[0,1]^{\mathbb Z}\bigr),
\]
let
\[
x^0,x^1\in X,
\qquad
t^0,t^1\in[0,1]^{\mathbb Z}
\]
be such that
\[
z^0=\Phi(x^0,t^0),
\qquad
z^1=\Phi(x^1,t^1).
\]
Set
\[
u:=x^0|_{[-N,N]},
\qquad
v:=x^1|_{[-N,N]},
\]
and define the following cylinders in \(X\):
\[
A_0:=[u]=\{x\in X:x|_{[-N,N]}=u\},
\qquad
A_1:=[v]=\{x\in X:x|_{[-N,N]}=v\}.
\]
Both are nonempty open sets since they correspond to the points
\(x^0,x^1\in X\).

By the positive-density independence consequence of
Theorem~\ref{kerr-li-entropy-ie} stated above, applied to the nonempty
open cylinder sets
\[
A_0=[u],
\qquad
A_1=[v],
\]
there exists an independence set \(J\subset\mathbb Z_{\geq0}\) for
\((A_0,A_1)\) with positive upper density. Hence there exist \(c_0>0\) and
a sequence \(M_r\to\infty\) such that
\[
|J\cap\{0,1,\ldots,M_r-1\}|\geq c_0M_r.
\]
For each \(r\), take
\[
\{n_1<\dots<n_k\}\subseteq J\cap\{0,1,\dots,M_r-1\}
\]
with gaps
\[
n_{i+1}-n_i>2N+1,
\]
as in the proof of Theorem~\ref{Invariant measure of full support implies CPMD}.
Note that it implies
\[
k\geq \frac{|J\cap\{0,1,\dots,M_r-1\}|}{2N+2}-1.
\]
Therefore, taking large enough \(r\) and
\[
c=\frac{c_0}{2(2N+2)}>0,
\]
yields
\[
k\geq cM_r.
\]
Since subsets of independence sets remain independent,
\(\{n_1,\dots,n_k\}\) is still independent for \((A_0,A_1)\).

Now define a path
\[
h:[0,1]\rightarrow \overline{\mathbb D}^{[-N,N]}
\]
by
\[
h(s)_j=
\begin{cases}
(1-3s)t^0_jp_{u_j}, & 0\leq s\leq \frac13,\\[2mm]
0, & \frac13\leq s\leq \frac23,\\[2mm]
(3s-2)t^1_jp_{v_j}, & \frac23\leq s\leq 1,
\end{cases}
\qquad -N\leq j\leq N.
\]
Thus
\[
h(0)=z^0|_{[-N,N]},
\qquad
h(1)=z^1|_{[-N,N]}.
\]
Note that by the definition, this is a continuous path passing through \(0\).

Define
\[
\Psi:[0,1]^k\rightarrow \overline{\mathbb D}^{\mathbb Z}
\]
coordinatewise as follows. Since the windows
\[
[n_i-N,n_i+N],
\qquad
1\leq i\leq k,
\]
are pairwise disjoint, the following formula is well-defined:
\[
\Psi(s)_m=
\begin{cases}
h(s_i)_j, & \text{if }m=n_i+j\text{ for some }1\leq i\leq k
\text{ and }-N\leq j\leq N,\\[2mm]
0, & \text{otherwise}.
\end{cases}
\]
By this explicit coordinate definition, \(\Psi\) is continuous as a map
\[
\Psi:[0,1]^k\rightarrow \overline{\mathbb D}^{\mathbb Z}.
\]

We now prove that \(\Psi(s)\in Z\) for every \(s\in[0,1]^k\).
Fix \[s=(s_1,\ldots,s_k)\in[0,1]^k\] and define two subsets of the selected
time set \(\{n_1,\ldots,n_k\}\) by
\[
F_0(s)=\{n_i:s_i\in[0,1/3)\},
\qquad
F_1(s)=\{n_i:s_i\in(2/3,1]\}.
\]
Thus \(F_0(s)\) corresponds to those windows where the inserted block lies on
the \(u\)-spokes, and \(F_1(s)\) corresponds to those where the inserted
block lies on the \(v\)-spokes. For indices with \(s_i\in[1/3,2/3]\), the
inserted block is zero, so no symbolic word needs to be prescribed.

By independence, there exists
\[
x\in
\bigcap_{n_i\in F_0(s)}\sigma^{-n_i}A_0
\cap
\bigcap_{n_j\in F_1(s)}\sigma^{-n_j}A_1.
\]
Equivalently,
\[
\sigma^{n_i}x\in A_0 \quad (n_i\in F_0(s)),
\qquad
\sigma^{n_j}x\in A_1 \quad (n_j\in F_1(s)).
\]
Since
\[
A_0=[u]=\{y\in X:y_j=u_j\text{ for }-N\leq j\leq N\},
\]
the condition \(\sigma^{n_i}x\in A_0\) means
\[
(\sigma^{n_i}x)_j=u_j
\qquad
(-N\leq j\leq N).
\]
This is the same as
\[
x_{n_i+j}=u_j
\qquad
(-N\leq j\leq N).
\]
Similarly, if \(n_i\in F_1(s)\), then
\[
x_{n_i+j}=v_j
\qquad
(-N\leq j\leq N).
\]
Thus \(x\) has the word \(u\) on the window
\([n_i-N,n_i+N]\) for every \(n_i\in F_0(s)\), and the word \(v\)
on the window \([n_i-N,n_i+N]\) for every \(n_i\in F_1(s)\).

Define \(t\in[0,1]^{\mathbb Z}\) as follows. If \(m=n_i+j\) for some
\(1\leq i\leq k\) and \(-N\leq j\leq N\), set
\[
t_m=
\begin{cases}
(1-3s_i)t^0_j, & 0\leq s_i<\frac13,\\[2mm]
0, & \frac13\leq s_i\leq \frac23,\\[2mm]
(3s_i-2)t^1_j, & \frac23<s_i\leq 1.
\end{cases}
\]
For all other \(m\), set \(t_m=0\). By construction,
\[
\Psi(s)_m=t_mp_{x_m}
\qquad
\text{for every }m\in\mathbb Z.
\]
Hence
\[
\Psi(s)=\Phi(x,t)\in Z.
\]
Therefore \(\Psi\) is a continuous map
\[
\Psi:[0,1]^k\to Z.
\]

Set
\[
\beta=\bigvee_{i=1}^k\sigma^{-n_i}\cU.
\]
Consider the pulled-back cover \(\Psi^{-1}\beta\) of \([0,1]^k\).

Fix \(i\). If \(s_i=0\), then the \([-N,N]\)-block of
\(\sigma^{n_i}\Psi(s)\) equals \(z^0|_{[-N,N]}\). Hence
\[
\sigma^{n_i}\Psi(s)\in C^0\subset Z\backslash U,
\]
so
\[
\Psi(s)\notin\sigma^{-n_i}U.
\]
Similarly, if \(s_i=1\), then the \([-N,N]\)-block of
\(\sigma^{n_i}\Psi(s)\) equals \(z^1|_{[-N,N]}\). Hence
\[
\sigma^{n_i}\Psi(s)\in C^1\subset Z\backslash V,
\]
so
\[
\Psi(s)\notin\sigma^{-n_i}V.
\]

Every element of \(\Psi^{-1}\beta\) is contained in a set of the form
\[
\Psi^{-1}\left(\bigcap_{i=1}^k\sigma^{-n_i}B_i\right),
\qquad
B_i\in\{U,V\}.
\]
Fix \(i\). If \(B_i=U\), then this pulled-back set is disjoint from
the face \(\{s_i=0\}\). If \(B_i=V\), then this pulled-back set is disjoint
from the face \(\{s_i=1\}\). Hence no element of \(\Psi^{-1}\beta\) meets
both opposite faces \(\{s_i=0\}\) and \(\{s_i=1\}\) of \([0,1]^k\), for any \(i\).

Moreover, every finite open refinement \(\gamma\succ\Psi^{-1}\beta\)
inherits this property: no element of \(\gamma\) meets two opposite faces
of \([0,1]^k\). Therefore, by Lemma
\ref{bound of covering number of with no opposing faces},
\[
\operatorname{ord}(\gamma)\geq k
\]
for every finite open refinement \(\gamma\succ\Psi^{-1}\beta\). Taking the
minimum over all such refinements gives
\[
\mathcal D\big(\Psi^{-1}\beta\big)\geq k.
\]
By the continuity of \(\Psi\),
\[
\mathcal D(\Psi^{-1}\beta)\leq \mathcal D(\beta),
\]
therefore
\[
\mathcal D(\beta)\geq k.
\]

Since
\[
\{n_1,\dots,n_k\}\subseteq \{0,1,\dots,M_r-1\},
\]
the full join
\[
\cU_0^{M_r-1}
=
\bigvee_{m=0}^{M_r-1}\sigma^{-m}\cU
\]
refines
\[
\beta=\bigvee_{i=1}^k\sigma^{-n_i}\cU.
\]
By monotonicity of \(\mathcal D\) with respect to refinements, we have
\[
\mathcal D\big(\cU_0^{M_r-1}\big)\geq \mathcal D(\beta)\geq k\geq cM_r.
\]
Thus we have along a subsequence
\[
\frac{\mathcal D\big(\cU_0^{M_r-1}\big)}{M_r}\geq c,
\]
which implies
\[
\mdim(Z,\sigma,\cU)
=
\lim_{n\rightarrow\infty}
\frac{\mathcal D\big(\cU_0^{n-1}\big)}{n}
\geq
\limsup_{r\rightarrow\infty}
\frac{\mathcal D\big(\cU_0^{M_r-1}\big)}{M_r}
\geq c>0.
\]
Since \(\cU\) was an arbitrary standard open cover, \((Z,\sigma)\) has UPMD.
\end{proof}
\section{A family of systems with CPMD but not UPMD}

\subsection{Rotation coding}

Let
\[
R:\mathbb T^d\to\mathbb T^d,
\qquad
R(v)=v+a,
\]
where
\[
a=(a_1,\ldots,a_d),
\]
and assume that
\[
1,a_1,\ldots,a_d
\]
are rationally independent. Then \(R\) is minimal and uniquely
ergodic, with the unique invariant probability measure given by the
Haar measure \(m_{\mathbb T^d}\).

Let
\[
q:\mathbb R^d\to\mathbb T^d=\mathbb R^d/\mathbb Z^d
\]
be the quotient map. Let
\[
U_1,\ldots,U_k\subset\mathbb T^d
\]
be pairwise disjoint nonempty open sets, and put
\[
P_i=\overline{U_i}.
\]
We assume that
\[
\mathbb T^d=\bigcup_{i=1}^k P_i,
\qquad
k\geq2,
\]
and that the sets $U_i$ are regular in the sense that
\[
U_i=\operatorname{int}(P_i)
\qquad(1\leq i\leq k).
\]
In particular,
\[
\mathbb T^d\setminus\bigcup_{i=1}^k\partial P_i
=
\bigsqcup_{i=1}^k U_i.
\]

We also assume that for every
\(i\), there is a compact convex set
\[
\widetilde P_i\subset\mathbb R^d
\]
such that
\[
q(\widetilde P_i)=P_i
\qquad\text{and}\qquad
\operatorname{diam}(\widetilde P_i)<\frac12.
\]
In particular \(q\vert_{\widetilde P_i}\) is one-to-one. Moreover, for all
\(i,j\),
\[
\bigl((\widetilde P_i-\widetilde P_i)
+
(\widetilde P_j-\widetilde P_j)\bigr)
\cap\mathbb Z^d
=
\{0\}.
\]
Indeed, every element of the first set has norm \(<1\), whereas every nonzero
element of \(\mathbb Z^d\) has norm at least \(1\).

Set
\[
B=\bigcup_{i=1}^k\partial P_i.
\]
Then \(B\) has Haar measure zero and empty interior. Define
\[
\mathbb T^d_{\rm reg}
=
\mathbb T^d\setminus\bigcup_{m\in\mathbb Z}R^{-m}B.
\]
Then \(\mathbb T^d_{\rm reg}\) has full Haar measure and is dense.

For \(v\in\mathbb T^d_{\rm reg}\), the point \(R^m v\) lies in exactly one
open atom \(U_i\) for every \(m\in\mathbb Z\). Define
\[
c(v)\in\{1,\ldots,k\}^{\mathbb Z}
\]
by
\[
c(v)_m=i
\quad\Longleftrightarrow\quad
R^m v\in U_i.
\]
Finally set
\[
Y=\overline{c(\mathbb T^d_{\rm reg})}
\subset\{1,\ldots,k\}^{\mathbb Z}.
\]
Since \(\mathbb T^d_{\rm reg}\) is \(R\)-invariant and
\[
c(Rv)=\sigma c(v),
\]
the set \(Y\) is closed and shift invariant.

We shall repeatedly use the following immediate consequence of the definition
of \(Y\)

\begin{itemize}
    \item [($\star$)] if \(F\subset\mathbb Z\) is finite and a word
\((i_m)_{m\in F}\) appears in some point of~\(Y\), then
\[
\bigcap_{m\in F}R^{-m}U_{i_m}\neq\varnothing.
\]
\end{itemize}
Indeed, the corresponding cylinder is open and intersects
\(\overline{c(\mathbb T^d_{\rm reg})}\), hence it intersects
\(c(\mathbb T^d_{\rm reg})\).

\begin{Lemma}
\label{regular-Y-full-support-measure}
The symbolic system \((Y,\sigma)\) carries a shift-invariant Borel probability
measure of full support.
\end{Lemma}

\begin{proof}
Choose \(y^*\in Y\), and define a Borel map
\[
\kappa:\mathbb T^d\to Y
\]
by
\[
\kappa(v)=
\begin{cases}
c(v), & v\in\mathbb T^d_{\rm reg},\\
y^*, & v\notin\mathbb T^d_{\rm reg}.
\end{cases}
\]
Let
\[
\nu=\kappa_*m_{\mathbb T^d}.
\]
Since \(\mathbb T^d_{\rm reg}\) is \(R\)-invariant and has full Haar measure,
we have
\[
\kappa\circ R=\sigma\circ\kappa
\]
for Haar-a.e. \(v\). Therefore \(\nu\) is shift invariant.

It remains to show that \(\nu\) has full support. Let
\[
[i_0,\ldots,i_{n-1}]
\]
be a nonempty cylinder in \(Y\). By observation ($\star$),
\[
O=\bigcap_{r=0}^{n-1}R^{-r}U_{i_r}
\]
is nonempty and open. Hence
\[
m_{\mathbb T^d}(O)>0.
\]
Since \(\mathbb T^d_{\rm reg}\) has full Haar measure,
\[
m_{\mathbb T^d}(O\cap\mathbb T^d_{\rm reg})>0.
\]
For every \(v\in O\cap\mathbb T^d_{\rm reg}\), we have
\[
c(v)_r=i_r
\qquad(0\leq r<n).
\]
Thus
\[
\nu([i_0,\ldots,i_{n-1}])
\geq
m_{\mathbb T^d}(O\cap\mathbb T^d_{\rm reg})
>0.
\]
Therefore every nonempty cylinder in \(Y\) has positive \(\nu\)-measure, so
\(\nu\) has full support.
\end{proof}

\subsection{Lift bookkeeping}

%We shall use the Steiner point as a canonical point in a compact convex set.
%For a nonempty compact convex set \(K\subset\mathbb R^d\), define
%\[
%h_K(u)=\sup_{x\in K}\langle x,u\rangle
%\qquad(u\in S^{d-1}),
%\]
%and
%\[
%s(K)=d\int_{S^{d-1}}h_K(u)u\,d\lambda(u),
%\]
%where \(\lambda\) is normalized surface measure on \(S^{d-1}\). We use only
%two standard properties:
%\[
%s(K+b)=s(K)+b
%\qquad(b\in\mathbb R^d),
%\]
%and, if \(K\) has nonempty interior, then
%\[
%s(K)\in\operatorname{int}(K).
%\]
For every compact set $K \subset\mathbb R^d$, let $c(K)$ denote its center of mass (convex average of the points of $K$ with respect to the uniform measure on $K$). Then
\[
c(K+b)=c(K)+b
\qquad(b\in\mathbb R^d),
\]
and, if \(K\) has nonempty interior, then
\[
c(K)\in\operatorname{int}(K).
\]

For a finite word
\[
\omega\in\{*,1,\ldots,k\}^{\{0,\ldots,n-1\}},
\]
we call a coordinate \(r\) effective if \(\omega_r\neq *\). Define
\[
B_r(\omega)=
\begin{cases}
\mathbb T^d, & \omega_r=*,\\
P_{\omega_r}, & \omega_r\neq *,
\end{cases}
\]
and define the associated cell of \(\omega \) by:
\[
C_\omega:=\bigcap_{r=0}^{n-1}R^{-r}B_r(\omega).
\]
The corresponding open cell is
\[
C_\omega^\circ
=
\bigcap_{\omega_r\neq *}R^{-r}U_{\omega_r}.
\]

\begin{Lemma}[Lift bookkeeping]
\label{lift-bookkeeping}
The following facts hold.

\begin{enumerate}
\item Let
\[
C=\bigcap_{r=1}^s R^{-m_r}P_{i_r}
\]
be a nonempty closed cell, and assume that its corresponding open cell
\[
C^\circ=\bigcap_{r=1}^s R^{-m_r}U_{i_r}
\]
is nonempty. Then every connected component of \(q^{-1}(C)\) is a compact
convex set with nonempty interior. Moreover, \(q\) maps each such component
homeomorphically onto \(C\), and any two components differ by an element of
\(\mathbb Z^d\).

\item Let \(\{C_\alpha\}_{\alpha\in(0,1)}\) be a family of
nonempty cells satisfying
$C_\alpha^\circ\neq\varnothing$, 
such that each \(C_\alpha\) has at least one effective constraint, so that 
\[
0<\alpha\leq\beta<1
\quad\Longrightarrow\quad
C_\alpha\subset C_\beta,
\]
and without an infinite chain of properly decreasing sets.

After choosing one lifted component \(\widetilde C_\beta\) of
\(q^{-1}(C_\beta)\), there is a well defined nested family of lifted components
\[
\{\widetilde C_\alpha\}_{\alpha\in(0,1)}
\]
anchored at \(\widetilde C_\beta\) so that
\[
0<\alpha\leq\gamma<1
\quad\Longrightarrow\quad
\widetilde C_\alpha\subset \widetilde C_\gamma.
\]

\item Any two nested lift choices for the same family differ by one common
element of \(\mathbb Z^d\). Consequently,
\[
c(\{C_\alpha\}):=q\left(\int_0^1 c(\widetilde C_\alpha)\,d\alpha\right)
\]
is independent of the choice of $\widetilde C_\beta$.

%\item Fix \(n\geq1\), and let \(B_0\subset\mathbb R^d\) be bounded. Then there
%is a bounded set \(B_1\subset\mathbb R^d\) such that every lifted component of
%a length-\(n\) cell \(C_\omega\) with at least one effective coordinate and
%\(C_\omega^\circ\neq\varnothing\), if it meets \(B_0\), is contained in \(B_1\).
\end{enumerate}
\end{Lemma}

\begin{proof}
A lift of \(R^{-m}P_i\) has the form
\[
\widetilde P_i-ma+p,
\qquad p\in\mathbb Z^d.
\]
We first record uniqueness of compatible lifts. Suppose that a lift
\[
\widetilde P_i-ma+p
\]
meets both
\[
\widetilde P_j-\ell a+q
\qquad\text{and}\qquad
\widetilde P_j-\ell a+q'.
\]
Then
\[
q-q'\in
(\widetilde P_i-\widetilde P_i)
+
(\widetilde P_j-\widetilde P_j)
\]
and also \(q-q'\in\mathbb Z^d\). By the small-diameter condition,
\[
q-q'=0.
\]
Thus, once one lifted constraint is fixed, every other constraint has at most
one compatible lift.

Now let \(C=\bigcap_{r=1}^sR^{-m_r}P_{i_r}\) be as in the first item. A
connected component of \(q^{-1}(C)\) is contained in one lift of the first
constraint. The compatibility uniqueness just proved then determines at most
one compatible lift of every other constraint. Hence the component is exactly
an intersection of compact convex lifted constraints:
\[
\bigcap_{r=1}^s
\left(\widetilde P_{i_r}-m_ra+p_r\right).
\]
Thus it is compact and convex. Since \(C^\circ\neq\varnothing\), the
corresponding intersection of interiors is nonempty, so the component has
nonempty interior.

The map \(q\) maps this lifted component onto \(C\), because every point of
\(C\) has a unique lift in the chosen lift of the first constraint, and the
compatibility uniqueness forces that lift to satisfy all other constraints. It
is one-to-one on the component because it is one-to-one on the chosen lift of
the first constraint. Therefore \(q\) maps the component homeomorphically onto
\(C\). Finally, any two choices of the first lifted constraint differ by an
element of \(\mathbb Z^d\), and compatibility forces the same translate for all
other constraints. Thus any two components differ by an element of
\(\mathbb Z^d\). This proves the first item.

For the second item, since the family \(\{C_\alpha\}_{\alpha\in(0,1)}\) is
finite and nested, list the distinct cells as
\[
D_1\subsetneq D_2\subsetneq\cdots\subsetneq D_s.
\]
Let \(D_r=C_\beta\), and set
\[
\widetilde D_r=\widetilde C_\beta.
\]
Suppose first that \(a<r\), and assume that \(\widetilde D_{a+1}\) has already
been defined. Since
\[
D_a\subset D_{a+1},
\]
and \(q\) maps \(\widetilde D_{a+1}\) homeomorphically onto \(D_{a+1}\), the
set
\[
\widetilde D_{a+1}\cap q^{-1}(D_a)
\]
is nonempty and maps homeomorphically onto \(D_a\). It is contained in a single
connected component of \(q^{-1}(D_a)\). Since every component of \(q^{-1}(D_a)\)
maps homeomorphically onto \(D_a\), this intersection is exactly that component.
Define
\[
\widetilde D_a=
\widetilde D_{a+1}\cap q^{-1}(D_a).
\]
Continuing downward defines
\[
\widetilde D_1\subset\cdots\subset\widetilde D_r.
\]

Now suppose that \(a>r\), and assume that \(\widetilde D_{a-1}\) has already
been defined. Since
\[
D_{a-1}\subset D_a,
\]
the connected set \(\widetilde D_{a-1}\) is contained in a unique connected
component of \(q^{-1}(D_a)\). Define \(\widetilde D_a\) to be that component.
Continuing upward defines
\[
\widetilde D_r\subset\cdots\subset\widetilde D_s.
\]
Finally, set
\[
\widetilde C_\alpha=\widetilde D_a
\qquad
\text{whenever } C_\alpha=D_a.
\]
This gives the desired nested lifted family and proves the second item.

For the third item, let
\[
\{\widetilde C_\alpha\}
\qquad\text{and}\qquad
\{\widehat C_\alpha\}
\]
be two nested lift choices. By the first item, for each \(\alpha\) there is
\(p(\alpha)\in\mathbb Z^d\) such that
\[
\widehat C_\alpha=\widetilde C_\alpha+p(\alpha).
\]
If \(0<\alpha\leq\beta<1\), then nesting gives
\[
\widetilde C_\alpha+p(\alpha)
\subset
\widetilde C_\beta+p(\beta).
\]
But also
\[
\widetilde C_\alpha+p(\alpha)
\subset
\widetilde C_\beta+p(\alpha).
\]
The two sets
\[
\widetilde C_\beta+p(\alpha)
\qquad\text{and}\qquad
\widetilde C_\beta+p(\beta)
\]
are components of \(q^{-1}(C_\beta)\), and they intersect. Hence they are
equal, so \(p(\alpha)=p(\beta)\). Therefore \(p(\alpha)\) is independent of
\(\alpha\); call the common value \(p\). Using translation equivariance of center of mass,
\[
c(\widehat C_\alpha)=c(\widetilde C_\alpha+p)=c(\widetilde C_\alpha)+p.
\]
Hence
\[
\int_0^1c(\widehat C_\alpha)\,d\alpha
=
\int_0^1c(\widetilde C_\alpha)\,d\alpha+p,
\]
and applying \(q\) gives the desired independence.

%Finally, for fixed \(n\), there are only finitely many words
%\[
%\omega\in\{*,1,\ldots,k\}^{\{0,\ldots,n-1\}}.
%\]
%For each word with at least one effective coordinate and
%\(C_\omega^\circ\neq\varnothing\), the first item says that the components of
%\(q^{-1}(C_\omega)\) are all integer translates of one compact convex shape.
%Since there are only finitely many such shapes, and since \(B_0\) is bounded,
%only finitely many integer translates of these shapes can meet \(B_0\). Their
%union is contained in some bounded set \(B_1\). This proves the fourth item.
\end{proof}

\subsection{The spoke cover and the cone map}

Set
\[
Z=\Spoke(Y).
\]
Let \(p_1,\ldots,p_k\) be the spoke directions used in the definition of
\(\Spoke(Y)\), and let
\[
S=\bigcup_{i=1}^k p_i[0,1]
\]
be the one-coordinate spoke set. Thus every \(z\in Z\) can be written as
\[
z=\Phi(y,\rho),
\qquad
y\in Y,\quad \rho\in[0,1]^{\mathbb Z},
\]
where
\[
z_m=\rho_m p_{y_m}.
\]

For \(1\leq i\leq k\), define
\[
A_i=\{w\in S: |w|<1/2\text{ or }w\in p_i[0,1]\},
\]
and
\[
\widehat U_i=\{z\in Z:z_0\in A_i\}.
\]
Let
\[
\widehat {\cU}=\{\widehat U_1,\ldots,\widehat U_k\}.
\]
Each \(\widehat U_i\) is open in \(Z\), because the complement of \(A_i\) in
\(S\) is the closed set
\[
\bigcup_{\ell\neq i}p_\ell[1/2,1].
\]

Fix \(n\geq1\). For \(z\in Z\), set
\[
M(z,n)=\max_{0\leq r<n}|z_r|,
\qquad
t(z,n)=\min\{1/2,M(z,n)\}.
\]
If \(t(z,n)>0\), then for \(\alpha\in(0,1)\) define
\[
\omega^\alpha(z,n)\in\{*,1,\ldots,k\}^{\{0,\ldots,n-1\}}
\]
as follows. For \(0\leq r<n\), put
\[
\omega_r^\alpha(z,n)=*
\quad\Longleftrightarrow\quad
|z_r|<t(z,n)\alpha.
\]
If
\[
|z_r|\geq t(z,n)\alpha,
\]
then \(z_r\neq0\), and so \(z_r\) lies on a unique spoke \(p_i(0,1]\). In this
case set
\[
\omega_r^\alpha(z,n)=i.
\]

Write \(z=\Phi(y,\rho)\). If
\(\omega_r^\alpha(z,n)=i\), then \(z_r\neq0\), so the spoke label is uniquely
determined and \(y_r=i\). Therefore the effective coordinates of
\(\omega^\alpha(z,n)\) form a finite word appearing in \(y\in Y\). By the construction and the above, we clearly have
\[
C_{\omega^\alpha(z,n)}^\circ\neq\varnothing.
\]
Moreover, if \(0<\alpha\leq\beta<1\), then every coordinate effective for
\(\omega^\beta(z,n)\) is also effective for \(\omega^\alpha(z,n)\), with the
same label. Hence
\[
C_{\omega^\alpha(z,n)}
\subset
C_{\omega^\beta(z,n)}.
\]
Finally, if \(r_0\) satisfies
\[
|z_{r_0}|=M(z,n),
\]
then \(r_0\) is effective for every \(\alpha\in(0,1)\). Hence if $t(z,n)>0$, all cells
\(C_{\omega^\alpha(z,n)}\) have at least one effective constraint.

Let
\[
\operatorname{Cone}(\mathbb T^d)
=
\mathbb T^d\times[0,1/2]\big/\sim,
\]
where
\[
(x,0)\sim(x',0)
\qquad(x,x'\in\mathbb T^d).
\]
We write its points as \([x,t]\), with \(x\in\mathbb T^d\) and
\(0\leq t\leq1/2\). Let \(0_{\mathbb T^d}\) denote the identity element of
\(\mathbb T^d\).

We now define
\[
\phi_n:Z\to\operatorname{Cone}(\mathbb T^d).
\]
If
\[
t(z,n)=0,
\]
set
\[
\phi_n(z)=[0_{\mathbb T^d},0].
\]
Assume now that
\[
t(z,n)>0.
\]
Write
\[
C_\alpha=C_{\omega^\alpha(z,n)}.
\]
By the above, \(\{C_\alpha\}_{\alpha\in(0,1)}\) is a
finite nested family of nonempty cells, every \(C_\alpha\) has nonempty
open cell, and every \(C_\alpha\) has at least one effective constraint. We may now use Lemma~\ref{lift-bookkeeping}.(3) to define a point $x(z,n)\in \T^d$ by 
\[
x(z,n)=
c(\{ C_\alpha\}).
\]
Set
\[
\phi_n(z)=[x(z,n),t(z,n)].
\]

\begin{Proposition}[The cone map]
\label{cone-map-continuity}
For every \(n\geq1\), the map
\[
\phi_n:Z\to\operatorname{Cone}(\mathbb T^d)
\]
is well-defined and continuous.
\end{Proposition}

\begin{proof}
The preceding discussion proves that \(\phi_n\) is well-defined. It remains to
prove continuity.

Let
\[
z^{(m)}\to z
\]
in \(Z\). If
\[
t(z,n)=0,
\]
then
\[
t(z^{(m)},n)\to0,
\]
so \(\phi_n(z^{(m)})\) converges to the cone vertex.

Assume now that
\[
t:=t(z,n)>0.
\]
Set
\[
t_m=t(z^{(m)},n).
\]
Then
\[
t_m\to t.
\]
For \(0\leq r<n\), define
\[
\theta_r=\frac{|z_r|}{t},
\qquad
\theta_r^{(m)}=\frac{|z_r^{(m)}|}{t_m}.
\]
Then
\[
\theta_r^{(m)}\to\theta_r.
\]

Choose a nested lift family
\[
\{\widetilde C_\alpha\}_{\alpha\in(0,1)}
\]
for the cells \(C_{\omega^\alpha(z,n)}\), and set
\[
F(\alpha)=c(\widetilde C_\alpha).
\]
Since only finitely many lifted cells occur, their union is contained in a
bounded set \(B_0\subset\mathbb R^d\).

Fix \(\eta>0\) small enough that
\[
6n\eta<1.
\]
Define
\[
G(\eta)
=
(0,1)\setminus
\bigcup_{r=0}^{n-1}
\bigl((\theta_r-3\eta,\theta_r+3\eta)\cap(0,1)\bigr),
\]
and let
\[
E(\eta)=(0,1)\setminus G(\eta).
\]
Then \(G(\eta)\neq\varnothing\) and
\[
|E(\eta)|\leq6n\eta.
\]
Choose
\[
\beta\in G(\eta).
\]

For all sufficiently large \(m\), and for every \(\alpha\in G(\eta)\), the
inequalities
\[
|z_r|<t\alpha
\qquad\text{and}\qquad
|z_r^{(m)}|<t_m\alpha
\]
agree for all \(0\leq r<n\). Indeed, these inequalities are equivalent to
\[
\theta_r<\alpha
\qquad\text{and}\qquad
\theta_r^{(m)}<\alpha,
\]
and \(\alpha\) stays away from the thresholds \(\theta_r\). If \(z_r\neq0\),
then the spoke label of \(z_r^{(m)}\) is eventually the same as the spoke label
of \(z_r\). If \(z_r=0\), then for \(\alpha\in G(\eta)\) the coordinate is
starred for all sufficiently large \(m\). Therefore
\[
\omega^\alpha(z^{(m)},n)=\omega^\alpha(z,n)
\qquad(\alpha\in G(\eta))
\]
for all sufficiently large \(m\).

In particular,
\[
C_{\omega^\beta(z^{(m)},n)}
=
C_{\omega^\beta(z,n)}
\]
for all sufficiently large \(m\). Thus we may use the already chosen lifted
component \(\widetilde C_\beta\) as the anchor for the nested lift family
associated to \(z^{(m)}\). Let
\[
\{\widetilde C_\alpha^{(m)}\}_{\alpha\in(0,1)}
\]
be the corresponding nested lift family, and set
\[
F_m(\alpha)=c(\widetilde C_\alpha^{(m)}).
\]
For \(\alpha\in G(\eta)\), the cells for \(z\) and \(z^{(m)}\) agree, and the
two nested lift families have the same anchor. Hence
\[
F_m(\alpha)=F(\alpha)
\qquad(\alpha\in G(\eta))
\]
for all sufficiently large \(m\).

It remains to estimate the integral over \(E(\eta)\). Every lifted cell in the
anchored families for \(z^{(m)}\) intersects the anchor \(\widetilde C_\beta\),
and
\[
\widetilde C_\beta\subset B_0.
\]
By Lemma~\ref{lift-bookkeeping}, there is a bounded set
\(B_1\subset\mathbb R^d\), depending only on \(B_0\) and \(n\), containing all
lifted cells occurring in these anchored families. Hence there is \(L>0\) such
that
\[
|F(\alpha)|\leq L,
\qquad
|F_m(\alpha)|\leq L
\]
for all relevant \(\alpha\) and all large \(m\). Therefore
\[
\left|
\int_0^1F_m(\alpha)\,d\alpha
-
\int_0^1F(\alpha)\,d\alpha
\right|
\leq
2L|E(\eta)|
\leq
12Ln\eta.
\]
Letting first \(m\to\infty\) and then \(\eta\to0\), we get
\[
q\left(\int_0^1F_m(\alpha)\,d\alpha\right)
\to
q\left(\int_0^1F(\alpha)\,d\alpha\right).
\]
Since \(t_m\to t\), this proves
\[
\phi_n(z^{(m)})\to\phi_n(z).
\]
Thus \(\phi_n\) is continuous.
\end{proof}

\subsection{Compatibility and bounded dimension growth}

\begin{Proposition}[Compatibility of the cone map]
\label{cone-map-compatibility}
For every \(n\geq1\),
\[
\mathcal D({\widehat \cU}_0^{n-1})\leq d+1.
\]
Consequently,
\[
\mdim(Z,\sigma,\widehat \cU)=0.
\]
\end{Proposition}

\begin{proof}
Fix \(n\geq1\). By Proposition~\ref{cone-map-continuity}, the map
\[
\phi_n:Z\to\operatorname{Cone}(\mathbb T^d)
\]
is continuous.

We first prove the key detection property. Suppose
\[
t(z,n)=1/2,
\qquad
\phi_n(z)=[x,1/2],
\]
and suppose that for some \(0\leq r<n\) and some \(i\),
\[
z_r\in p_i[1/2,1].
\]
Then, for every \(\alpha\in(0,1)\), the coordinate \(r\) is effective in
\(\omega^\alpha(z,n)\), with label \(i\). Hence
\[
C_{\omega^\alpha(z,n)}\subset R^{-r}P_i
\qquad(\alpha\in(0,1)).
\]
Choose the nested lift family used in the definition of \(\phi_n(z)\). Since
the lifted cells are nested, all these lifted cells lie in one common lift of
\(R^{-r}P_i\). Their center of mass lie in the interiors of the lifted cells,
hence in the interior of this common lift of \(R^{-r}P_i\). Since this interior
is convex,
\[
\int_0^1c(\widetilde C_\alpha)\,d\alpha
\]
also lies in the interior of the common lift. Therefore
\[
x\in R^{-r}U_i,
\]
or equivalently,
\[
R^r x\in U_i.
\]

We now show that every fiber of \(\phi_n\) is contained in one element of
\(\widehat \cU_0^{n-1}\). Let
\[
Q\in \phi_n(Z),
\qquad
F_Q=\phi_n^{-1}(Q).
\]
If \(Q\) has cone height \(<1/2\), then every \(z\in F_Q\) satisfies
\[
|z_r|<1/2
\qquad(0\leq r<n).
\]
Thus \(z_r\in A_i\) for every \(i\) and every \(0\leq r<n\). Hence \(F_Q\) is
contained in any element of \(\widehat \cU_0^{n-1}\).

Now suppose that
\[
Q=[x,1/2].
\]
Fix \(0\leq r<n\). If every \(z\in F_Q\) satisfies
\[
|z_r|<1/2,
\]
choose an arbitrary label \(\ell_r\). Otherwise, choose \(z\in F_Q\) and \(i\)
such that
\[
z_r\in p_i[1/2,1].
\]
By the detection property,
\[
R^r x\in U_i.
\]
If \(z'\in F_Q\) also satisfies
\[
z'_r\in p_j[1/2,1],
\]
then again
\[
R^r x\in U_j.
\]
Since the sets \(U_1,\ldots,U_k\) are pairwise disjoint, we get \(i=j\). Thus
the visible label at coordinate \(r\) is uniquely determined by \(Q\). Call it
\(\ell_r\).

With this choice of labels, every \(z\in F_Q\) satisfies
\[
z_r\in A_{\ell_r}
\qquad(0\leq r<n).
\]
Therefore
\[
F_Q
\subset
\bigcap_{r=0}^{n-1}\sigma^{-r}\widehat U_{\ell_r}.
\]
Thus every fiber of \(\phi_n\) is contained in one element of
\(\widehat \cU_0^{n-1}\).

By a standard compactness argument, if a continuous map from a compact
space to a compact metrizable space has every fiber contained in some element
of a finite open cover, then the map is compatible with that cover (see \cite[Prop.\ 2.3]{Lindenstrauss-Weiss}). Applying
this to \(\phi_n\), we see that \(\phi_n\) is compatible with
\(\widehat \cU_0^{n-1}\).

Since
\[
\dim\operatorname{Cone}(\mathbb T^d)\leq d+1,
\]
we obtain
\[
\mathcal D(\widehat \cU_0^{n-1})\leq d+1,
\]
hence
\(
\mdim(Z,\sigma,\widehat \cU)
=0\).
\end{proof}

\subsection{Concluding remarks}
To conclude, we give some additional properties of the construction studied in this section.

\begin{Theorem}
\label{regular-spoke-family-CPMD-not-UPMD}
For the system \((Z,\sigma)\) constructed above, the following hold:
\begin{enumerate}
\item \((Z,\sigma)\) has CPMD.
\item \((Z,\sigma)\) is not UPMD.
\end{enumerate}
More precisely, there is a standard open cover
\[
\beta=\{W_0,W_1\}
\]
such that
\[
\mdim(Z,\sigma,\beta)=0.
\]
Moreover,
\[
\mathcal N(\beta_0^{n-1})\geq n+1
\qquad(n\geq1).
\]
If, in addition, the elements of the partition on the torus are given by rectangles with sides parallel to the standard vectors, 
then there is \(C>0\) such that
\[
\mathcal N(\beta_0^{n-1})\leq Cn^d
\qquad(n\geq1),
\]
and hence
\[
h_{\mathrm{top}}(Z,\sigma,\beta)=0.
\]
In this case
\((Z,\sigma)\) is neither UPMD nor UPE.
\end{Theorem}

\begin{proof}
    
By Lemma~\ref{regular-Y-full-support-measure}, the symbolic system
\((Y,\sigma)\) carries a shift-invariant Borel probability measure of full
support. Therefore, by Theorem~\ref{Invariant measure of full support implies CPMD},
\[
(Z,\sigma)=\Spoke(Y)
\]
has CPMD. To construct such a partition $\beta$,
choose any
\(
\varnothing\neq U_j\neq\mathbb T^d
\),
and set
\[
W_0=\widehat U_j,
\qquad
W_1=\bigcup_{\ell\neq j}\widehat U_\ell,
\qquad
\beta=\{W_0,W_1\}.
\]
Then \(\beta\) is a standard open cover of \(Z\). 

Moreover, \(\beta\) is coarser than \(\widehat \cU\). Hence
\[
\mathcal D(\beta_0^{n-1})
\leq
\mathcal D(\widehat \cU_0^{n-1})
\leq d+1
\qquad(n\geq1),
\]
and therefore
\[
\mdim(Z,\sigma,\beta)=0.
\]
However, we claim that $\mathcal N(\beta_0^{n-1})\to\infty$. Let
\[
\pi_j:Y\to\{0,1\}^{\mathbb Z}
\]
be the binary factor defined by
\[
\pi_j(y)_m=
\begin{cases}
0, & y_m=j,\\
1, & y_m\neq j.
\end{cases}
\]
Set
\[
X_j=\pi_j(Y).
\]

We first note that \(X_j\) is nonperiodic. Indeed, if the regular binary coding
of some \(v\in\mathbb T^d_{\rm reg}\) had period \(q\geq1\), then along the
dense orbit of \(v\) the condition of lying in \(U_j\) would be unchanged by
applying \(R^q\). Hence \(R^q(U_j)\subset P_j\), and since \(R^q(U_j)\) is open
and \(U_j=\operatorname{int}(P_j)\), we get
\[
R^q(U_j)\subset U_j.
\]
Applying the same argument to \(R^{-q}\) gives
\[
R^q(U_j)=U_j.
\]
This is impossible, since \(R^q\) is minimal and \(U_j\) is a nonempty proper
open set. Thus \(X_j\) is nonperiodic.

Since \(X_j\) is nonperiodic, a standard word-complexity bound gives
\[
p_{X_j}(n)\geq n+1
\qquad(n\geq1),
\]
where \(p_{X_j}(n)\) is the number of length-\(n\) words appearing in \(X_j\).
Now take any length-\(n\) word in \(X_j\), realize it by some \(y\in Y\), and
choose all radii equal to \(1\). The corresponding point of \(Z\) belongs to
exactly one element of 
\(
\beta_0^{n-1},
\)
and that element is determined by the binary word. Distinct binary words give
distinct indispensable elements of this cover. Therefore
\[
\mathcal N(\beta_0^{n-1})
\geq
p_{X_j}(n)
\geq
n+1.
\]

If we assume that the partition on the torus we are starting from is by rectangles parallel to the standard directions,  the same coding argument
also gives
\[
\mathcal N(\beta_0^{n-1})\leq Cn^d.
\]
\end{proof}

\printbibliography
\end{document}